\documentclass[11pt,a4paper]{article}
\usepackage{latexsym,epsfig}
\usepackage{amsmath,amssymb,amsthm}
\usepackage{graphicx}
\usepackage{enumerate}
\usepackage{url,hyperref}
\usepackage[margin=1.1in]{geometry}
\usepackage{comment}
\usepackage{subfig}

\long\def\ignore#1{}

\def\R{{\cal R}}

\graphicspath{{.},{figures/}}

\newtheorem{theorem}{Theorem}
\newtheorem{lemma}{Lemma}[section]

\newtheorem{corollary}[theorem]{Corollary}
\newtheorem{proposition}[lemma]{Proposition}
\newtheorem{claim}[lemma]{Claim}

\def\qed{\ifvmode\mbox{ }\else\unskip\fi\hskip 1em plus 10fill$\Box$}

\newcommand{\bbox}{\vrule height7pt width4pt depth1pt}

\newcommand{\ceil}[1]{\left\lceil #1 \right\rceil}

\begin{document}
\title{On Partitions of Two-Dimensional Discrete Boxes}

\author{
Eyal Ackerman\thanks{
Department of Mathematics, Physics, and Computer Science, 
University of Haifa at Oranim, Tivon 36006, Israel.
{\tt ackerman@sci.haifa.ac.il}.
}
\and
Rom Pinchasi\thanks{
Mathematics Department,
Technion---Israel Institute of Technology,
Haifa 32000, Israel.
{\tt room@technion.ac.il}. Supported by ISF grant (grant No.\ 1091/21)}
}

\maketitle
\begin{abstract}
Let $A$ and $B$ be finite sets and consider a partition of the \emph{discrete box} $A \times B$ into \emph{sub-boxes} of the form $A' \times B'$ where $A' \subset A$ and $B' \subset B$.
We say that such a partition has the \emph{$(k,\ell)$-piercing} property 
for positive integers $k$ and $\ell$ if for every $a \in A$
the \emph{discrete line} $\{a\} \times B$ intersects at least $k$ sub-boxes and for every $b \in B$ the line $A \times \{b\}$ intersects at least $\ell$ sub-boxes. We show that a partition of $A \times B$ that has the 
$(k, \ell)$-piercing property must consist of at least 
$(k-1)+(\ell-1)+\left\lceil 2\sqrt{(k-1)(\ell-1)} \right\rceil$ sub-boxes. This bound is nearly tight (up to one additive unit) for all values of $k$ and $\ell$
and is tight for infinitely many values of $k$ and $\ell$.

As a corollary we get that the same bound holds for the minimum number of vertices of a graph whose edges can be colored red and blue such that every vertex is part of a red $k$-clique and a blue $\ell$-clique.
\end{abstract}

\section{Introduction}

Consider the following puzzle: Let $k$ be a positive integer and suppose that an axes-parallel rectangle $R$ in the plane is partitioned into $n$ rectangles such that every axis-parallel line that intersects $R$ intersects at least $k$ of these rectangles. Then how small can $n$ be as a function of $k$?

Denote this function by $n(k)$ and observe that $n(1)=1$ and $n(k)=4k-4$ for $k>1$.
Indeed, the two lines that contain the top and bottom sides of $R$ intersect together $2k$ distinct rectangles when $k>1$. Similarly, the two lines that contain the left and right sides of $R$ intersect $2k$ distinct rectangles. There are exactly four rectangles that belong to these two sets --- the ones containing the four corners of $R$ --- hence $n(k) \geq 4k-4$.
To see that this bound is tight consider the example in Figure~\ref{fig:bricks} (taken from~\cite{Bucic}).
\begin{figure}[ht]
	\centering
	\includegraphics[width=6cm]{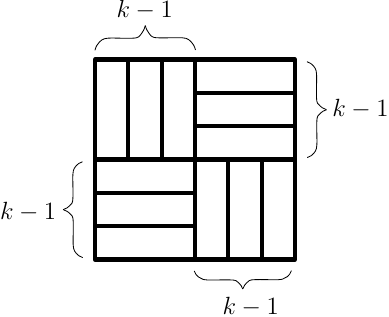}
	\caption{A partition into $4k-4$ rectangles with the $k$-piercing property.}
	\label{fig:bricks}
\end{figure}
For an extension of this problem to three dimensions see~\cite{3D}. 

This puzzle becomes non-trivial when instead of geometric rectangles one considers \emph{discrete boxes}.
A $d$-dimensional \emph{discrete box} $D$ is a set of the form $A_1 \times A_2 \times \ldots \times A_d$ where each $A_i$ is a finite set of size at least two.
A set of the form $A'_1 \times A_2' \times \ldots \times A'_d$ such that $A'_i \subseteq A_i$ for each $i \in [d]$ is called a \emph{sub-box} of $D$.
We say that a family of sub-boxes \emph{partitions} $D$ if every member of $D$ is contained in exactly one sub-box.
A family of sub-boxes has the \emph{$k$-piercing} property if every \emph{discrete line} intersects at least $k$ sub-boxes, where a discrete line is a set of the form $A'_1 \times A_2' \times \ldots \times A'_d$ where for some $i \in [d]$ 
we have $A'_i=A_i$ and for every $j \in [d] \setminus \{i\}$ we have $A'_j=\{a_j\}$ for some $a_j \in A_j$.

Bucic, Lidicky, Long and Wagner~\cite{Bucic} asked for the minimum size of a family of sub-boxes that partitions a $d$-dimensional discrete box and has the $k$-piercing property. They denoted this number by $p_{\rm box}(d,k)$ and showed that $e^{\Omega(\sqrt{d})}k \leq p_{\rm box}(d,k) \leq 15^{d/2}k$.
It follows from a result of Alon, Bohman, Holzman and Kleitman~\cite{Alon} that  $p_{\rm box}(d,2) = 2^d$.
Considering the two-dimensional case, Bucic et al.~\cite{Bucic}  proved that $p_{\rm box}(2,k) \ge (4-o_k(1))k$, observed that $p_{\rm box}(2,k) \leq 4k-4$ (by the example in Figure~\ref{fig:bricks}) and conjectured that the latter is a tight bound.
Their conjecture was settled by Holzman~\cite{Holzman} using a reduction to edge-coloring of graphs that was suggested by Bucic et al.~\cite{Bucic}. Namely, he proved that if the edges of a graph can be two-colored such that every vertex belongs to a monochromatic $k$-clique of each color, then the graph has at least $4k-4$ vertices.
In fact, in his proof Holzman has (implicitly) reduced the problem on edge-colored graphs back to the problem on pierced boxes and thus showed that these two problems are equivalent. Moreover, all the graphs for which the bound is tight were characterized in~\cite{Holzman}.

In this paper we focus on the \emph{asymmetric} two-dimensional case. Namely, we say that a family of sub-boxes that partitions a two-dimensional discrete box $A \times B$ has the \emph{$(k,\ell)$-piercing} property if 
every \emph{row} in $A \times B$ intersects at least $k$ sub-boxes and every \emph{column} in $A \times B$ intersects at least $\ell$ sub-boxes.
By a \emph{row} we mean a discrete line of the form $\{a\} \times B$ for some $ a \in A$ and by a \emph{column} we mean a discrete line of the form $A \times \{b\}$ for some $b \in B$.

It is easy to generalize the above-mentioned arguments and conclude that in the geometric case (that is, where the boxes are actual rectangles pierced by vertical or horizontal lines) the number of boxes is at least $2k+2\ell-4$ and that this bound is tight.
However, as opposed to the symmetric case, in the asymmetric case we get a better bound when considering discrete boxes.

\begin{theorem}\label{thm:main}
	For every $k,\ell \ge 2$, every family of sub-boxes that partitions a discrete box and has the $(k,\ell)$-piercing property contains at least 
$(k-1)+(\ell-1)+\left\lceil 2\sqrt{(k-1)(\ell-1)} \right\rceil$ sub-boxes.
	Moreover, for every $k,\ell \ge 2$ there is a family of $(k-1)+(\ell-1)+2\left\lceil \sqrt{(k-1)(\ell-1)} \right\rceil$ sub-boxes that partitions a two-dimensional discrete box and has the $(k,\ell)$-piercing property.
\end{theorem}

Note that the lower bound and the upper bound in Theorem~\ref{thm:main} differ by at most one unit and they coincide for an infinite number of distinct values of $k$ and $\ell$.
We also remark that our proof differs from the proof in~\cite{Holzman} and is somewhat simpler.
From the reduction mentioned above we immediately get:

\begin{corollary}\label{cor:graph}
	Let $k>1$ and $\ell>1$ be positive integers and let $G$ be a graph whose edges can be colored with red and blue such that every vertex belongs to a red $k$-clique and a blue $\ell$-clique. Then $G$ has at least 
$(k-1)+(\ell-1)+\left\lceil 2\sqrt{(k-1)(\ell-1)} \right\rceil$ vertices. This bound is nearly sharp for every $k$ and $\ell$ and sharp for an infinite number of distinct values of $k$ and $\ell$.
\end{corollary} 

\section{Proof of Theorem~\ref{thm:main}}

We prove the first part of Theorem~\ref{thm:main} in Section~\ref{sec:lower} and then describe the construction that proves the second part of the theorem in Section~\ref{sec:upper}.

\subsection{The lower bound}
\label{sec:lower}
Suppose for contradiction that the first part of the theorem is false. That is, there is a non-empty set of counter-examples each of which is a family of sub-boxes that partition a two-dimensional discrete box $A \times B$, has the $(k,\ell)$-property for some $k,\ell \ge 2$ and consists of less than $(k-1)+(\ell-1)+\left\lceil 2\sqrt{(k-1)(\ell-1)} \right\rceil$ sub-boxes.
Among these counter-examples consider the ones with the minimum sum $k+\ell$, among those consider the ones with the minimum sum $|A|+|B|$ and among those let $\cal R$ be a counter-example with the least number of sub-boxes.
Suppose that $\cal R$ consists of less than $(k-1)+(\ell-1)+\left\lceil 2\sqrt{(k-1)(\ell-1)} \right\rceil$ sub-boxes with the $(k,\ell)$-property that partition the discrete box $A \times B$ and
assume without loss of generality that $A=[m]$ and $B=[n]$.

It is easy to verify that at least one of $k$ and $\ell$ must be greater than two and we leave it as a small exercise to the reader in order to get a sense of the problem. Note also that the case $k=\ell=2$ is a direct consequence of the results in~\cite{Alon} and~\cite{Holzman}.

Let $A' \times B'$ be a sub-box. 
If $|A'|=1$ or $|B'|=1$, then we say that $A' \times B'$ is \emph{thin}.
In the former case we call the sub-box
\emph{horizontally thin} whereas in the latter case it is \emph{vertically thin}). 
If $|A'|=|B'|=1$, then we say that $A' \times B'$ is a \emph{singleton}.

\begin{proposition}\label{prop:one}
	Every row and column contains a thin sub-box of $\cal R$.
\end{proposition}

\begin{proof}
	Suppose for example that row $i$ does not contain a (horizontally) 
	thin sub-box.
	Then by deleting this row, that is, by removing $i$ from $A$, 
	no sub-box is deleted and we remain with a partition of 
	$(A \setminus \{i\}) \times B$ that still has the $(k,\ell)$ property.
	However, this contradicts the minimality of $m+n$.
	In a similar way we can conclude that there is no column without a thin sub-box.
\end{proof}

It will be convenient to assume that if a row (resp., a column) contains 
several thin sub-boxes, then all of them but possibly one are singletons.
Indeed, suppose for example that there are several thin sub-boxes contained 
in row $i$, say, $\{i\} \times B'_1$, $\{i\} \times B'_2$, \ldots, $\{i\} \times B'_s$. 
Then for every $j=2, \ldots, s$ choose $b'_{j} \in B'_j$
and replace the original thin sub-boxes with the thin sub-boxes 
$\{i\} \times \left(B'_1 \cup \bigcup_{j=2}^s B'_j \setminus \{b'_j\}\right)$, 
$\{i\} \times \{b'_2\}$, \ldots, $\{i\} \times \{b'_s\}$.
Note that this results in another partition of $A \times B$ with the same 
number of sub-boxes and this partition still has the 
$(k,\ell)$-piercing property.
 
\begin{proposition}\label{prop:induction}
	If $k \geq \ell$ (resp., $\ell \geq k$), then every column (resp., row) contains exactly one thin sub-box of $\cal R$.
\end{proposition}

\begin{proof}
Assume without loss of generality that $k \geq \ell$ and column $j$ contains at least two thin sub-boxes.
By deleting column $j$, that is, by replacing $B$ with $B \setminus \{j\}$ 
we obtain a partition of $A \times (B \setminus \{j\})$ that has the 
$(k-1,\ell)$-piercing property.
Recall that it is known and easy to prove that the lower bound holds for $k=\ell=2$ (as we mention above).
Therefore, we may assume that $k > 2$ and it follows from the minimality of $\cal R$ that $|{\cal R}| \geq 2 + (k-2) + (\ell-1) + \left\lceil 2\sqrt{(k-2)(\ell-1)} \right\rceil$.
Thus, to get a contradiction it remains to show that the following inequality holds:
\begin{equation}\label{eq:1}
\left\lceil 2\sqrt{(k-2)(\ell-1)} \right\rceil \geq 2\sqrt{(k-1)(\ell-1)}-1.
\end{equation}
If $k=\ell$, then~(\ref{eq:1}) holds since $\left\lceil 2\sqrt{(k-2)(k-1)} \right\rceil \geq 2(k-1)-1$ for every integer $k > 2$.
If $k > \ell$, then~(\ref{eq:1}) holds if $2\sqrt{\ell-1}(\sqrt{k-1}-\sqrt{k-2}) \leq 1$.
This inequality indeed holds since we have:
\begin{eqnarray}
2\sqrt{\ell-1}(\sqrt{k-1}-\sqrt{k-2})&=&2\sqrt{\ell-1}(\sqrt{k-1}-\sqrt{k-2})\frac{\sqrt{k-1}+\sqrt{k-2}}{\sqrt{k-1}+\sqrt{k-2}}= \nonumber \\ 
&=&\frac{2\sqrt{\ell-1}}{\sqrt{k-1}+\sqrt{k-2}} \leq \frac{\sqrt{k-1}+\sqrt{k-2}}{\sqrt{k-1}+\sqrt{k-2}} = 1, \nonumber
\end{eqnarray}
where the last inequality holds because $k > \ell$.
\end{proof}

\begin{corollary}\label{cor:1}
	If some row (resp., column) contains more than one thin sub-box, then every column (resp., row) contains exactly one thin sub-box.
\end{corollary}

\begin{proposition}\label{prop:induction2}
	If a row (resp., column) contains at least two thin sub-boxes, then not all of them are singletons.
\end{proposition}

\begin{proof}
Note that we may permute the rows and columns without breaking the $(k,\ell)$-property,
therefore we may assume without loss of generality that Row $1$ contains several singletons and no other thin sub-boxes and that these singletons are $\{(1,1)\}, \ldots, \{(1,s)\}$, $s \geq 2$.
It follows from Proposition~\ref{prop:induction} and Corollary~\ref{cor:1} that $k > \ell$ and none of the Columns $1,\ldots,s$ contains another thin sub-box. 
Note that Row $1$ intersects exactly $k$ sub-boxes, for otherwise by
deleting Column $1$ we can obtain a (smaller) family of sub-boxes $\cal R'$ with the $(k,\ell)$-piercing property, contradicting the minimality of $\cal R$.
It follows that $k \ge s$ and in fact $k>s$. Indeed, if $k=s$, then $n=s=k$ and each row must intersect only vertically thin sub-boxes. However, each column should contain at most one vertically thin sub-box.
	
Delete Row $1$ and Columns $1,\ldots,s$ and obtain a (smaller) family of sub-boxes $\cal R'$.
It follows from the minimality of $\cal R$  that $\cal R'$ does not have the $(k,\ell)$-piercing property.
Therefore, some (non-thin) boxes that were contained in the union of columns $1,\ldots,s$ were deleted and there is a row that now intersects less than $k$ sub-boxes.
Let $t \ge 1$ be the smallest integer such that $\cal R'$ has the $(k-t,\ell)$-piercing property.
Notice that $t \le s/2$ and that the union of the columns $1,\ldots,s$ contains at least $s+t$ sub-boxes.
Since $k > s \ge 2t \ge t+1$, it follows that $k-t \ge 2$.
Therefore, by the minimality of $\cal R$, we have $|{\cal R'}| \ge (k-t-1)+(\ell-1)+2\sqrt{(k-t-1)(\ell-1)}$.
Thus, $|{\cal R}| \ge s+t+|{\cal R'}| \ge s+(k-1)+(\ell-1)+2\sqrt{(k-t-1)(\ell-1)}$,
which leads to a contradiction if the last expression is at least $(k-1)+(\ell-1)+2\sqrt{(k-1)(\ell-1)}$. 
This happens if $2\sqrt{\ell-1}(\sqrt{k-1}-\sqrt{k-t-1}) \le s$ and indeed:
\begin{eqnarray}
  2\sqrt{\ell-1}(\sqrt{k-1}-\sqrt{k-t-1})=2\sqrt{\ell-1}(\sqrt{k-1}-\sqrt{k-t-1})\frac{\sqrt{k-1}+\sqrt{k-t-1}}{\sqrt{k-1}+\sqrt{k-t-1}}= \nonumber \\ 
  =\frac{2t\sqrt{\ell-1}}{\sqrt{k-1}+\sqrt{k-t-1}} \leq \frac{s\sqrt{\ell-1}}{\sqrt{k-1}+\sqrt{k-t-1}} < \frac{s\sqrt{\ell-1}}{\sqrt{k-1}} < s, \nonumber
\end{eqnarray}
where the last inequality holds because $k > \ell$.	
\end{proof}

\begin{proposition}
	There is no singleton that is the only thin sub-box in both row and column that contain it.
\end{proposition}

\begin{proof}
	Suppose for contradiction that $\cal R$ has a singleton 
	$\{(i,j)\}$ that is the only thin sub-box in Row $i$ and Column $j$. 
	If we remove Row $i$ and Column $j$, that is, replace $A$ with 
	$A \setminus \{i\}$, replace $B$ with $B \setminus \{j\}$ and change the sub-boxes in $\cal R$ accordingly, then we decrease the number of sub-boxes by exactly one since $\{(i,j)\}$ is the only thin sub-box in the Row $i$ and Column $j$.
	However, this implies that the new partition still has the $(k,\ell)$-piercing property
	which contradicts the minimality of $\cal R$.
\end{proof}

In summary, we may assume that $\cal R$ has the following properties: (1)~Every row (resp., column) contains a thin sub-box; (2)~if a row (resp., column) contains several thin sub-boxes then: all of them but one are singletons and every column (resp., row) contains exactly one thin sub-box; and (3)~there is no singleton which is the only thin sub-box both in its row and in its column.

Next we associate every row and every column with a unique thin sub-box that is contained in that row or column as follows.
If a row or a column contains a non-singleton thin sub-box, then we assign this sub-box to that row or column.
If Row $i$ contains a singleton $\{(i,j)\}$ and no other thin sub-box, then by the properties above, Column $j$ must contain a non-singleton thin sub-box which is associated to it. Therefore, we can assign the singleton $\{(i,j)\}$ to Row $i$.
Similarly, if Column $j$ contains only one thin sub-box which is a singleton, then we can assign this singleton to Column $j$. See Figure~\ref{fig:x_i} for an example.
\begin{figure}[t]
	\centering
	\includegraphics[width=6cm]{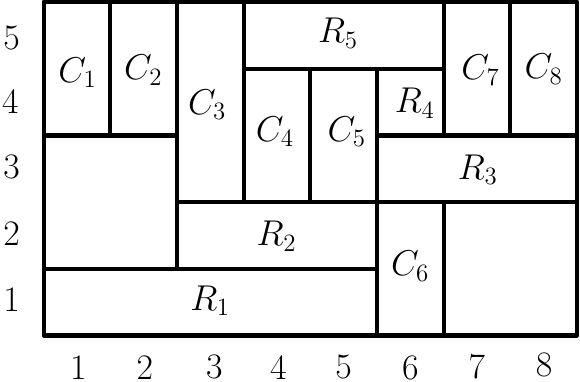}
	\caption{A (geometric) partition. Row $i$ is associated with $R_i$. Column $j$ is associated with $C_j$. The number of columns whose associated thin sub-box is not intersected by Row $2$ is $x_3=5$ (Columns 1,2,6,7,8). The number of rows that do not intersect $C_7$ is $y_7=3$ (Rows 1--3). The number of non-thin sub-boxes that Row 2 intersects is $t_2=2$.}
	\label{fig:x_i}
\end{figure}
We conclude that the number of thin sub-boxes is at least $m+n$, and hence, $|{\cal R}| \geq m+n$ (recall that $A=[m]$ and $B=[n]$).

\begin{proposition}\label{prop:assoc}
Suppose that Row $i$ intersects a vertically thin sub-box $R$ that is contained in Column $j$.
Then either $R$ is associated with Column $j$ or $R$ is a singleton which is associated with Row $i$. 
\end{proposition}

\begin{proof}
Suppose that $R$ is not associated with Column $j$.
Then it must be a singleton by Property~(2) mentioned above and the way vertically thin boxes are associated with columns.
Furthermore, $R$ is the only thin box in Row $i$ and is associated with it.  
\end{proof}

For every $i \in [m]$ let $x_i$ denote the number of columns whose associated thin sub-boxes are not intersected by Row $i$. 
Let $t_i$ denote the number of non-thin sub-boxes that intersect Row $i$.
For every $j \in [n]$ denote by $y_j$ the number of rows that do not intersect the thin sub-box that is associated with Column $j$. 
See Figure~\ref{fig:x_i} for an example.

Since the $x_i$'s and the $y_j$'s both count pairs of a row and column whose vertically thin sub-box is not intersected by that row, by double counting we get:
\begin{equation}\label{eq:double-counting}
	\sum_{i=1}^m x_i = \sum_{j=1}^n y_j.
\end{equation}

For a sub-box $S = A' \times B'$ we denote $a_{S} = |A'|$ and $b_{S} = |B'|$. Thus,
\begin{equation}\label{eq:sum-t_i}
\sum_{i=1}^m t_i = \sum_{S \in {\cal R} \mid a_{S},b_{S} \geq 2} a_{S}.
\end{equation}

Consider Column $j$ and let $R$ be the vertically thin sub-box that is associated with it. Apart from $R$ there are at least $(\ell-1)$ other sub-boxes that intersect Column $j$.
Each such sub-box $S$ is a witness for $a_{S}$ rows that do not intersect $R$.
Hence, 

\begin{equation}\label{eq:aS}
\displaystyle{y_{j} \geq \ell-1 + \sum_{S \mid S \neq R, \textrm{~~$S$ intersects column 
$j$}}(a_{S}-1).}
\end{equation}

Notice that such a box $S$ contributes $a_{S}-1$ to the right
hand side of (\ref{eq:aS}) for exactly $b_{S}$ different columns $j$.
Therefore by summing over all the columns we have:
\begin{equation}\label{eq:sum-y_j}
\sum_{j=1}^n y_j \geq n(\ell-1) + \sum_{S \in {\cal R} \mid a_{S},b_{S} \geq 2} (a_{S}-1)b_{S}.
\end{equation}

Consider Row $i$ and let us try to bound from below the number $n$ of
columns in $A \times B$.
Row $i$ intersects the horizontally thin sub-box assigned to it and another $t_{i}$ sub-boxes that are not vertically thin in their 
column. Therefore, because Row $i$ intersects at least $k$ sub-boxes in $\R$, 
it must intersect at least $k-(t_{i}+1)$ sub-boxes each of which is vertically
thin in its column. 
By Proposition~\ref{prop:assoc} each of these vertically thin sub-boxes is associated with the column it belongs to. 
In addition there are $x_{i}$ columns whose vertically thin sub-boxes are not intersected by Row $i$. We may therefore conclude that
$n \geq x_{i}+k-1-t_{i}$. 

Summing over all rows and using~(\ref{eq:double-counting}), (\ref{eq:sum-t_i}) and (\ref{eq:sum-y_j}) we have
\begin{eqnarray}
	n &\geq& \frac{1}{m} \sum_{i=1}^m \left( x_i+k-1-t_i\right) = k-1+\frac{1}{m} \left(\sum_{j=1}^n y_j - \sum_{i=1}^m t_i\right) \nonumber \\
	&\geq& k-1+\frac{n}{m}(\ell-1) +\frac{1}{m}\sum_{S \in {\cal R} \mid a_{S},b_{S} \geq 2} \left((a_{S}-1)b_{S}-a_{S}\right) \nonumber \\
	&\geq& k-1+\frac{n}{m}(\ell-1), \label{eq:n}
\end{eqnarray}

where the last inequality holds since $(a_{S}-1)b_{S} \geq a_{S}$ for $a_{S}, b_{S} \geq 2$.
By symmetry, we get that $m \geq \ell-1+\frac{m}{n}(k-1)$.
Combining this with (\ref{eq:n}) we obtain 
\begin{eqnarray}\label{eq:n+m}
|{\cal R}| \geq n+m &\geq& (k-1) + (\ell-1) + \frac{n}{m}(\ell-1) + \frac{m}{n}(k-1) \nonumber \\
&\geq& (k-1) + (\ell-1) + 2\sqrt{(k-1)(\ell-1)}, \nonumber 
\end{eqnarray}

where the last inequality follows by observing that $x^2+y^2 \ge 2xy$ for every $x$ and $y$ and setting $x=\sqrt{\frac{n}{m}(\ell-1)}$ and $y=\sqrt{\frac{m}{n}(k-1)}$
(equality is attained when $\frac{m}{n}=\frac{\sqrt{\ell-1}}{\sqrt{k-1}}$).

Therefore, $|\R| \geq (k-1)+(\ell-1)+\left\lceil 2\sqrt{(k-1)(\ell-1)} \right\rceil$.
This leads to a contradiction and thus completes the proof of the first part of Theorem~\ref{thm:main}.


\paragraph{Remark.}
Inequality~(\ref{eq:n}) can be refined to $n \geq k-1+\lceil \frac{n}{m}(\ell-1) \rceil$ and similarly $m \geq \ell-1+\lceil \frac{m}{n}(k-1) \rceil$.
From here we get 
$$
|\R| \geq m+n \geq (k-1) + (\ell-1) + 
\min_{m,n}\left\{\left\lceil \frac{n}{m}(\ell-1) \right\rceil + \left\lceil \frac{m}{n}(k-1) \right\rceil\right\}.
$$

In some cases this bound is better (by one additive unit) than $(k-1)+(\ell-1)+\left\lceil 2\sqrt{(k-1)(\ell-1)} \right\rceil$ and matches the upper bound construction described below.
Still, in other cases even this refined analysis does not match the upper bound construction. 


\subsection{The upper bound construction}
\label{sec:upper}

Suppose without loss of generality that $k \ge \ell \ge 2$.
%
We first describe a construction 
of a partition that meets the lower bound of 
$(k-1)+(\ell-1)+2\sqrt{(k-1)(\ell-1)}$ such that $\sqrt{(k-1)(\ell-1)}$ is an integer (in particular this includes the case $k = \ell$).
For such $k$ and $\ell$ the analysis of the construction is rather simple.
Later we describe how to modify the construction in a simple way 
for any $k \ge \ell \ge 2$.

It would be convenient to set $L=\sqrt{\ell-1}$ and $K=\sqrt{k-1}$ and to describe the construction geometrically.
Consider an axis-parallel rectangle $R$ 
whose bottom-left corner is at the origin,
whose width is $K^2+KL$ and whose height is $L^2+KL$.
Let $s_1$ be the line-segment whose endpoints are $(KL,0)$ and $(K^2+KL,KL)$.
Let $s_2$ be the line-segment whose endpoints are $(0,0)$ and $(K^2+KL,L^2+KL)$.
Let $s_3$ be the line-segment whose endpoints are $(0,KL)$ and $(KL,L^2+KL)$.
Thus, the slopes of $s_1$, $s_2$, and $s_3$ are $\frac{KL}{K^2}$, $\frac{L^2+KL}{K^2+KL}$ and $\frac{L^2}{KL}$ respectively, 
and therefore these three line segments are parallel.
Color the part of $R$ above $s_3$ light gray and do the same for 
the part of $R$ below $s_2$ and above $s_1$.
The remaining parts of $R$ we color dark gray 
(see Figure~\ref{fig:construction-perfect}(a) for an example).
\begin{figure}[t]
	\centering
	\subfloat[The initial coloring by light and dark gray.]{\includegraphics[width= 7cm]{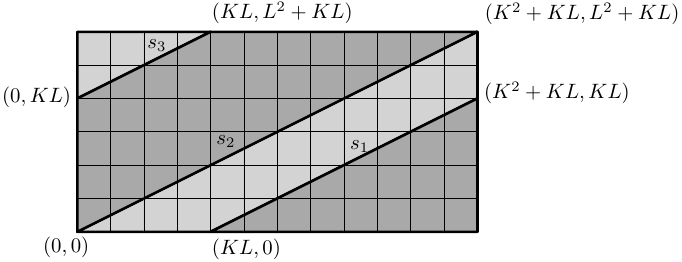}}
	\hspace{5mm}
	\subfloat[The partition into light gray horizontally thin sub-boxes and dark gray vertically thin sub-boxes.]{\includegraphics[width= 7cm]{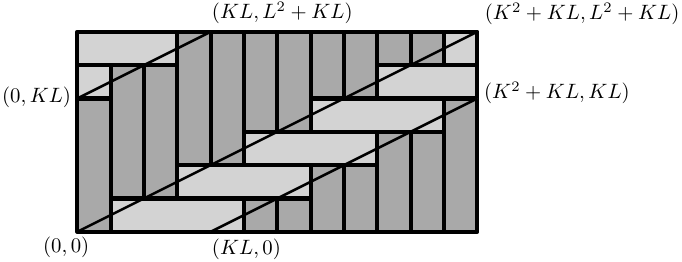}}
	\caption{The upper bound construction for $k=K^2+1=9$ and $\ell=L^2+1=3$.}
	\label{fig:construction-perfect}		
\end{figure}

The rectangle $R$ can be naturally partitioned into $(K^2+KL)\times(L^2+KL)$ unit squares, $\{\square_{i,j} \mid 0 \leq i <K^2+KL,~~0 \leq j <L^2+KL\}$, 
where $\square_{i,j}$ denotes the unit square whose bottom-left corner is at $(i,j)$. 
We color each of these unit squares either by light gray or by dark gray
according to the dominant color within that unit square in the initial 
coloring of the rectangle $R$ (see Figure~\ref{fig:construction-perfect}(b)). 
In case of a tie, that is, when some segment $s_i$ splits a unit square into 
two parts of equal area, we use the color of the part above $s_i$ 
(we note that because $KL \ge 2$, it is not possible that the same unit square is crossed by more than one of the segments $s_1$, $s_2$, and $s_3$).

The suggested coloring of the unit squares within the rectangle $R$
induce in a very natural way a partition of a combinatorial box
of dimensions $(K^2+KL)\times(L^2+KL)$ into $K^2+KL$ vertically-thin sub-boxes
and $L^2+KL$ horizontally-thin sub-boxes in the following way.
Combine the light gray squares within each `row' of unit squares 
into a horizontally thin sub-box and combine the dark gray squares within 
each `column' of unit squares into a vertically thin sub-box.
Note that this construction can be realized geometrically by rectangles drawn 
on a torus. Notice also that in this construction we interchange the roles
of rows and columns as defined in the introduction.
The reason of course is because when we fix the first coordinate
in the Euclidean plane we get a geometrically vertical line while
combinatorially we are used to thinking about a row as being horizontal. 
This is similar to the confusion when referring to the 
$(i,j)$ entry in a matrix to be in the $i$'th row and the $j$'th column,
as opposed to the way we think about the point with coordinates $(i,j)$
in the plane.

We claim that there are exactly $KL$ light
gray unit squares in every `row' of $R$ and there are exactly $KL$
dark gray unit squares in every `column' of $R$. This will follow from 
the following simple geometric observation.

\begin{claim}\label{claim:geom}
Let $P$ be a parallelogram of area $g \ge 1$ and height $1$ 
whose vertices are $(0,0)$,
$(g,0)$, $(t,1)$, and $(t+g,1)$ for some $0 < t \leq g$
(in other words we assume here that $P$ has two horizontal sides
whose projections on the $x$-axis overlap).
Let $U$ be an axes-parallel unit square whose center is $(c,\frac{1}{2})$. 
Then the area of $U \cap P$ is greater than $\frac{1}{2}$ if and only if 
$\frac{t}{2} < c < g+\frac{t}{2}$.
\end{claim}

\noindent {\bf Proof.} 
Notice that the area of $U \cap P$ is unimodal 
(increasing and then decreasing) in $c$. For $c=\frac{t}{2}$ and for 
$c=g+\frac{t}{2}$ the 
area of $U \cap P$ is precisely $\frac{1}{2}$ (here we use the fact that 
$g \geq 1$ and that the projections of the two horizontal side of $P$ on the 
$x$-axis overlap).
We leave it to the reader to verify the details.
\bbox

Consider now any row of unit squares in $R$. In the initial
coloring of $R$ the light gray area in each row is, up to a cyclic shift and a translation, 
a parallelogram
$P$ that satisfies the conditions in Claim \ref{claim:geom}.
To verify that indeed the projections of the horizontal edges of $P$
on the $x$-axis overlap one has just to verify that the slope of the line segment
$s_{1}$ is greater than or equal to $\frac{1}{KL}$. This is indeed true as the 
slope of $s_{1}$ is equal to $\frac{L}{K}$ and $L \geq 1$.
It follows now from Claim \ref{claim:geom} and from the assumption that 
$KL$ is an integer that there are precisely $KL$ light
gray unit squares in every 'row' of $R$. In precisely the same way one concludes
that there are precisely $KL$
dark gray unit squares in every 'column' of $R$. Here we apply
Claim \ref{claim:geom} on the vertical parallelograms in each column. We need to verify that the projections of the two vertical edges on the $y$-axis overlap.
To this end one has to verify that 
the slope of $s_{1}$ is not greater than $KL$. Once again this is true
because the slope of $s_{1}$ is equal to $\frac{L}{K}$ and $K \geq 1$.

Having shown that there are exactly $KL$ light
gray unit squares in every 'row' of $R$ and there are exactly $KL$
dark gray unit squares in every 'column' of $R$ we conclude
that there are $K^2+KL$ vertically-thin sub-boxes (one in each 'column')
and $L^2+KL$ horizontally-thin sub-boxes (one in each 'row') 
in our partition. Therefore, it is a partition of a discrete box into
$K^2+KL+L^2+KL=(k-1)+(\ell-1)+2\sqrt{(k-1)(\ell-1)}$ sub-boxes.

It remains to show that our partition 
has the $(K^2+1,L^2+1)$-piercing property.
Indeed, every row contains one horizontally-thin sub-box 
consisting of $KL$ unit squares and therefore the remaining
$K^2$ unit squares in this row belong to pairwise distinct vertically-thin
sub-boxes. Altogether every row intersects $K^2+1$ sub-boxes of our partition.
Similarly, every column intersects $L^2+1$ 
sub-boxes of our partition, one vertically-thin box and $L^2$ 
horizontally-thin sub-boxes, as desired.

\bigskip

\noindent {\bf The construction for general $k$ and $\ell$.}

Suppose that $\sqrt{(k-1)(\ell-1)}$ is not an integer and
assume without loss of generality that $k > \ell$.
As before, set $K=\sqrt{k-1}$ and $L=\sqrt{\ell-1}$.
We modify the construction above as follows.
Consider the axis-parallel rectangle $R$ 
whose bottom-left corner is at the origin,
whose width is $K^2+\lceil KL \rceil$ and whose height is $L^2+\ceil{KL}$.
Let $s_1$ be the line-segment whose endpoints are $(\ceil{KL},0)$ and $(K^2+ \ceil{KL},\ceil{KL})$.
Let $s_2$ be the line-segment whose endpoints are $(0,0)$ and $(\frac{K^2(L^2+\ceil{KL})}{\ceil{KL}},L^2+\ceil{KL})$.
Let $s_3$ be the line-segment whose endpoints are $(0,\ceil{KL})$ and $(\frac{K^2L^2}{\ceil{KL}},L^2+\ceil{KL})$.
Thus, $s_1$, $s_2$, and $s_3$ are parallel (see Figure~\ref{fig:construction-imperfect}).
\begin{figure}[t]
	\centering
	\includegraphics[width= 8cm]{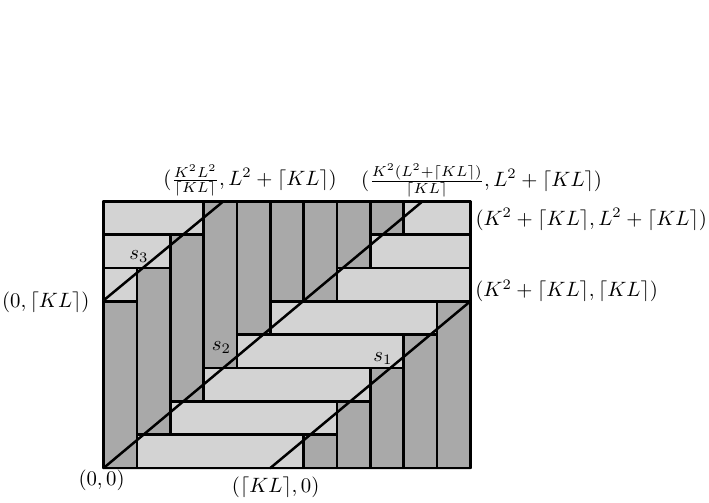}
	\caption{The upper bound construction for $k=K^2+1=7$ and $\ell=L^2+1=4$.}
	\label{fig:construction-imperfect}		
\end{figure}

The three line segments $s_{1}, s_{2}$ and $s_{3}$ 
partition $R$ into four regions that we color
by light gray and dark gray as before. Namely, we color by light gray the region
in $R$ above $s_{3}$ and the region in $R$ bounded between $s_{1}$ and $s_{2}$.
We color by dark gray the region in $R$ below $s_{1}$ and the region in $R$
bounded between $s_{2}$ and $s_{3}$.

As before, $R$ can be naturally partitioned into 
$(K^2+\ceil{KL})\times(L^2+\ceil{KL})$ unit squares, 
$\{\square_{i,j} \mid 0 \leq i <K^2+\ceil{KL},~~0 \leq j <L^2+\ceil{KL}\}$, 
where $\square_{i,j}$ denotes the unit square whose bottom-left corner is at $(i,j)$. 
We color each of these unit squares either by light gray or by dark gray
in precisely the same manner as we did before, that is, 
according to the dominant color in that unit square in the initial coloring
of the rectangle $R$.

The coloring of the unit squares in $R$ by light gray and dark gray induces
in a natural way a partition of the combinatorial box
of dimensions $(K^2+\ceil{KL}) \times (L^2+\ceil{KL})$ into 
$K^2+\ceil{KL}$ vertically-thin sub-boxes and $L^2+\ceil{KL}$ horizontally-thin
sub-boxes. As before, we combine the light gray squares within each `row' of 
unit squares 
into a horizontally thin sub-box and combine the dark gray squares within 
each `column' of unit squares into a vertically thin sub-box.
  
Here again the region colored light gray in every `row' of $R$ is, 
up to a cyclic shift and a translation, 
a parallelogram
$P$ that satisfies the conditions in Claim \ref{claim:geom}.
To verify that indeed the projections of the horizontal edges of $P$
on the $x$-axis overlap, one has just to verify that the slope of the line segment
$s_{1}$ is greater than or equal to $\frac{1}{\ceil{KL}}$. This is indeed true as the 
slope of $s_{1}$ is equal to $\frac{\ceil{KL}}{K^2}$ and $L \geq 1$.
It follows now from Claim \ref{claim:geom} and from the fact that 
$\ceil{KL}$ is an integer that there are precisely $\ceil{KL}$ light
gray unit squares in every `row' of $R$
(see Figure~\ref{fig:construction-imperfect}).
Therefore, every row contains one light gray sub-box and the rest of the
$K^2$ dark gray unit squares belong to pairwise distinct vertically thin 
sub-boxes. Altogether, every row intersects precisely $K^2+1$ sub-boxes
in our partition. 

Let us now consider the number of sub-boxes each column intersects.
We wish to show that this is at least $L^2+1$. 
We notice that the left end of $s_{1}$ is to the right of the right end of $s_{3}$. This is because $\frac{K^2L^2}{\ceil{KL}}  < \ceil{KL}$ (since $KL$ is not an integer).
It follows (and we leave the details to the reader) that the dark gray region within every vertical strip of width $1$
in $R$ is, up to a cyclic vertical shift, contained in a parallelogram $P$ of area $\ceil{KL}$ in our initial coloring of the rectangle 
$R$ (see Figure~\ref{fig:construction-imperfect}).
We claim that $P$ satisfies the conditions of Claim
\ref{claim:geom} in the sense that the projections of the two vertical 
edges of $P$ on the $y$-axis overlap. In order for this to be true
we need to verify that the slope of $s_{1}$ is not greater than $\ceil{KL}$.
This is indeed true because the 
slope of $s_{1}$ is equal to $\frac{\ceil{KL}}{K^2}$ and $K \geq 1$.

We can now conclude from Claim \ref{claim:geom} that 
every column of unit squares contains a vertically thin sub-box that consists 
of 
at most $\ceil{KL}$ unit squares. Consequently, every column 
of unit squares intersects at least $L^2$ horizontally thin sub-boxes.
We will conclude the proof once we show that every column of unit squares in 
$R$ 
contains at least one dark gray unit square. Then altogether 
every column of unit squares in $R$ intersects $L^2+1$ sub-boxes
in our partition, as desired.

In order to show that every column of unit squares in 
$R$ contains at least one dark gray unit square consider a 
vertical line $m$ that crosses $R$. Observe that the set of the light gray 
points on $m$ in the initial coloring of $R$ constitute, up to a vertical 
cyclic shift, a segment
of length at most $\ceil{KL}$ times the slope of $s_{1}$. 
The slope of $s_{1}$ is equal to $\frac{\ceil{KL}}{K^2}$ and therefore
the length of the light gray segment on $m$ is at most
$\frac{\ceil{KL}^2}{K^2}$. From here we conclude that the light gray
region in each column of $R$ in the initial coloring of $R$ is contained
in a parallelogram $P$ of area at most $\frac{\ceil{KL}^2}{K^2}$.
One can check that $P$ satisfies the conditions of Claim \ref{claim:geom}.
It follows now from Claim \ref{claim:geom} that there are at most
$\ceil{\frac{\ceil{KL}^2}{K^2}} \leq \frac{\ceil{KL}^2}{K^2}+1$ light
gray unit squares in each column of $R$.
Recall that there are $L^2+\ceil{KL}$ unit squares in every column of $R$. 
Hence it is enough to show that $L^2+\ceil{KL}-(\frac{(\ceil{KL})^2}{K^2}+1)$
is greater than or equal to $1$. Set $\alpha=\ceil{KL}-KL$. Then

\begin{eqnarray}
L^2+\ceil{KL}-\left(\frac{(\ceil{KL})^2}{K^2}+1\right) & = & 
(KL+\alpha)-1-\frac{(KL+\alpha)^2-(KL)^2}{K^2}\nonumber\\ 
& \geq & KL+\alpha-1-\frac{2\alpha KL+\alpha^2}{K^2}\nonumber\\
& \geq & KL-1-\frac{2\alpha KL}{K^2}\nonumber\\
& \geq & KL-3.
\end{eqnarray}

Hence we are done when $KL \geq 4$.

\paragraph{Acknowledgments}
We thank Ron Holzman for pointing out a mistake in an earlier version of this 
paper.
We also thank anonymous reviewers for their remarks that helped improving the presentation of the paper, in particular for a suggestion on simplifying the upper bound construction. 

\bibliographystyle{abbrv}

\end{document}